\pdfoutput=1
\documentclass[a4paper]{amsart}
\usepackage [autostyle, english = american]{csquotes}
\MakeOuterQuote{"} 
\usepackage{amsmath}
\usepackage{amsthm}
\usepackage{amssymb}
\usepackage{graphicx}
\usepackage{amsfonts}
\usepackage{hyperref}
\usepackage{bbm}
\usepackage{mathrsfs}
\usepackage{accents}
\usepackage{mathtools}	% Per \mathclap
\usepackage{color}
\usepackage[usenames,dvipsnames]{xcolor}
\usepackage{mparhack}	% To correct the margin notes , which are sometimes the wrong side.
\usepackage{comment}	% For long comments : \ begin {comment } ... \ end {comment }
\usepackage{marvosym}	% For the Lightning : \ Lightning
\usepackage{enumitem}	% To customize the lists
\usepackage{array}		    % To better format cells in tables .
\usepackage{todonotes}
\numberwithin{equation}{section}
\theoremstyle{plain}
\newtheorem{Prop}{Proposition}[section]
\newtheorem{Thm}[Prop]{Theorem}
\newtheorem*{Thm*}{Theorem}
\newtheorem{Lem}[Prop]{Lemma}
\newtheorem{Cor}[Prop]{Corollary}

\theoremstyle{definition}

\theoremstyle{remark}
\newtheorem{Rem}[Prop]{Remark}

\def\vint_#1{\mathchoice%
          {\mathop{\kern 0.2em\vrule width 0.6em height 0.69678ex
depth -0.58065ex
                  \kern -0.8em \intop}\nolimits_{\kern -0.4em#1}}%
          {\mathop{\kern 0.1em\vrule width 0.5em height 0.69678ex
depth -0.60387ex
                  \kern -0.6em \intop}\nolimits_{#1}}%
          {\mathop{\kern 0.1em\vrule width 0.5em height 0.69678ex
              depth -0.60387ex
                  \kern -0.6em \intop}\nolimits_{#1}}%
          {\mathop{\kern 0.1em\vrule width 0.5em height 0.69678ex
depth -0.60387ex
                  \kern -0.6em \intop}\nolimits_{#1}}}
\def\vintslides_#1{\mathchoice%
          {\mathop{\kern 0.1em\vrule width 0.5em height 0.697ex depth -0.581ex
                  \kern -0.6em \intop}\nolimits_{\kern -0.4em#1}}%
          {\mathop{\kern 0.1em\vrule width 0.3em height 0.697ex depth -0.604ex
                  \kern -0.4em \intop}\nolimits_{#1}}%
          {\mathop{\kern 0.1em\vrule width 0.3em height 0.697ex depth -0.604ex
                  \kern -0.4em \intop}\nolimits_{#1}}%
          {\mathop{\kern 0.1em\vrule width 0.3em height 0.697ex depth -0.604ex
                  \kern -0.4em \intop}\nolimits_{#1}}}

\newcommand{\intav}{\vint}%----------------------------------------------------------------------integral average
\newcommand{\aveint}[2]{\mathchoice%--------------------------------------------------------integral average on reals
          {\mathop{\kern 0.2em\vrule width 0.6em height 0.69678ex
depth -0.58065ex
                  \kern -0.8em \intop}\nolimits_{\kern -0.45em#1}^{#2}}%
          {\mathop{\kern 0.1em\vrule width 0.5em height 0.69678ex
depth -0.60387ex
                  \kern -0.6em \intop}\nolimits_{#1}^{#2}}%
          {\mathop{\kern 0.1em\vrule width 0.5em height 0.69678ex
depth -0.60387ex
                  \kern -0.6em \intop}\nolimits_{#1}^{#2}}%
          {\mathop{\kern 0.1em\vrule width 0.5em height 0.69678ex
depth -0.60387ex
                  \kern -0.6em \intop}\nolimits_{#1}^{#2}}}

\DeclareMathOperator{\dv}{div}

\DeclareMathOperator{\osc}{osc}

\DeclareMathOperator{\dist}{dist}

\newcommand{\set}[2]{\left\{#1 : #2\right\}}%---------------------set comprehension
\newcommand{\mns}{\setminus}%------------------------------set excision
\newcommand{\1}{\mathbbm{1}}%--------------------charactersitic function
\newcommand{\N}{\mathbb{N}}%----------------------------------------natural numbers
\newcommand{\R}{\mathbb{R}}%----------------------------------------real numbers
\newcommand{\del}{\partial}%------------------------------------------------topological boundary

\newcommand{\eps}{\varepsilon}%-----------------------------------------epsilon
\newcommand{\dx}{\, dx}%-----------------------------------------dx after integral
\newcommand{\loc}{\text{\rm loc}}%-----------------------------------loc as subscript
\newcommand{\Om}{\Omega}%--------------------------------------------------------usually domain
\newcommand{\inp}[2]{\big\langle #1,#2\big\rangle}%-----------------inner product/duality action
\newcommand{\gr}{\nabla}%-----------------------------------------------------gradient
\newcommand{\lap}{\Delta}%------------------------------------------------Laplacian
\newcommand{\hh}{\mathbb{H}}%------------------------------------------- Heisenberg group
\newcommand{\X}{\mathfrak{X} }%------------------------------------------Horizontal gradient
\newcommand{\Xu}{\X u}
\newcommand{\laph}{\lap\,_{\!_H}}%------------------------------------------Sub-Laplacian
\newcommand{\XX}{\X\X}%--------------------------------------------Horizontal Hessian
\newcommand{\dvh}{\dv_{\!_H}}%--------------------------------------------Horizontal divergence
\newcommand{\dhh}{d_{\hh^n}}%----------------------H metric
\newcommand{\normh}[1]{\|#1\|_{\hh^n}}%-------H  norm
\newcommand{\W}[2]{\boldsymbol{W}^\mu_{#1,#2}}%--------------------Wolff potential
\title[Quasi-linear equations on the Heisenberg Group]
{Regularity of inhomogeneous Quasi-linear equations on the Heisenberg Group}

\author[Shirsho Mukherjee]{Shirsho Mukherjee}
\address[S.\ Mukherjee]{Department of Mathematics, 
University of Bologna, Piazza di Porta San Donato 5,
  40126 - Bologna (BO), Italy.}
\email{shirsho.mukherjee2@unibo.it}

\author[Yannick Sire]{Yannick Sire}
\address[Y.\ Sire]{Department of Mathematics, 
Johns Hopkins University, 3400 N. Charles Street, Baltimore MD 21218, USA.}
\email{sire@math.jhu.edu}

\begin{document}

\begin{abstract}
We establish H\"older continuity of the horizontal gradient of weak solutions to quasi-linear $p$-Laplacian type non-homogeneous equations in the Heisenberg Group.
\end{abstract}

\maketitle
\setcounter{tocdepth}{1}
\phantomsection
%\addcontentsline{toc}{section}{Contents}
\tableofcontents

\section{Introduction}\label{sec:Introduction}
The $C^{1,\alpha}$ regularity of the $p$-Laplacian has been established earlier in, for instance \cite{Dib,Lewis,Tolk} in the Euclidean setting. 
Its {\sl sub-elliptic} analogue for homgeneous sub-elliptic equations of $p$-Laplacian type on the Heisenberg group, was unavailable until \cite{Zhong, Muk-Zhong}, in the last years.   
It is therefore natural to consider the case of regularity for the corresponding inhomogeneous equation and this is the purpose of the present contribution.

In this paper, we consider the equation
\begin{equation}\label{eq:maineq}
-\dvh a(x, \X u) = \mu \quad\text{in}\ \Om\subseteq\hh^n, 
\end{equation}
where $\Om$ is a domain and $\mu$ is a Radon measure with 
$|\mu|(\Om)<\infty$ and $\mu(\hh^n\mns\Om)=0$; hence the equation \eqref{eq:maineq} can be considered as defined in all of $\hh^n$. 
Here we denote $\X u  = (X_1u,\ldots, X_{2n}u)$ as the horizontal gradient of $u:\Om\to\R$, see Section \ref{sec:Prelim}. 

We shall take up the following structural assumptions throughout the paper: the continuous function $a:\Om\times\R^{2n}\to\R^{2n}$ is assumed to be 
$C^1$ in the gradient variable and satisfies the following structure condition 
for every $x,y\in\Om$ and $z,\xi\in \R^{2n}$, 
\begin{equation}\label{eq:str} 
\begin{aligned}
(|z|^2+s^2)^\frac{p-2}{2}|\xi|^2\leq 
\inp{D_z a(x,z)\xi}{\xi}&\leq L (|z|^2+s^2)^\frac{p-2}{2}|\xi|^2; \\
|a(x,z)-a(y,z)|&\leq L' |z|(|z|^2+s^2)^\frac{p-2}{2}|x-y|^\alpha,
\end{aligned}
\end{equation}
where $L,L'\geq 1, s\geq 0$, $\alpha\in(0,1]$ and $D_z a(x,z)$ is a symmetric matrix for every $x\in\Om$. The sub-elliptic $p$-Laplacian 
equation with measure data, given by 
\begin{equation}\label{eq:eq}
-\dvh(|\Xu|^{p-2}\X u) = \mu, 
\end{equation}
is a prototype of the equation \eqref{eq:maineq} with the condition  
\eqref{eq:str} for the case $s=0$. The weak solutions of \eqref{eq:maineq} are defined in horizontal Sobolev space $HW^{1,p}(\Om)$; the Lipschitz and  H\"older classes, denoted by same classical notations, are defined with respect to the CC-metric $(x,y)\mapsto d(x,y)$, see Section \ref{sec:Prelim} for details. We shall denote $Q=2n+2$ as the homogeneous dimension.
Now we state our main result.

\begin{Thm}\label{thm:c1au}
Let $u\in HW^{1,p}(\Om)$ be a weak solution of equation \eqref{eq:maineq} with $p\geq 2$ and a $C^1$-function $a:\Om\times\R^{2n}\to \R^{2n}$ satisfying the structure condition \eqref{eq:str}. If we have 
$\mu= f\in L^q_\loc(\Om)$ for some $q>Q$, then $\X u$ is locally H\"older continuous and there exists $c=c(n,p,L)>0$ and 
${\bar R}={\bar R}(n,p,L,L',\alpha,q,\dist(x_0,\del\Om))>0$ such that for any $x_0\in\Om,\ 0< R\leq \bar R$ and $x,y\in B_R(x_0)\subset\Om$, the estimate 
\begin{equation}\label{eq:ciau}
 |\X u(x)-\X u(y)| \leq c\, d(x,y)^\gamma \bigg(\intav_{B_R(x_0)}(|\X u|+s)\dx + \big\|f \big\|_{L^q(B_R(x_0))}^{1/(p-1)}\bigg),
\end{equation}
holds for some $\gamma=\gamma(n,p,L,\alpha, q)\in (0,1)$. In in particular, if $a(x,z)$ is independent of $x$, then \eqref{eq:ciau} holds 
for ${\bar R}={\bar R}(n,p,L,\dist(x_0,\del\Om))>0$ and $\gamma(n,p,L, q)\in (0,1)$. 
\end{Thm}

The proof of Theorem \ref{thm:c1au} in this paper, relies on novel techniques introduced 
by Duzaar-Mingione \cite{Duz-Min} based on sharp {\sl comparison estimates} of homogeneous equations with frozen coefficients, in other words {\sl harmonic replacements}. 
However, in the present sub-elliptic setting,
one encounters extra terms coming from commutators of the horizontal vector fields which leads to estimates that are not always as strong as that in the Euclidean setting. An instance appears in 
Proposition \ref{prop:intosc0} for the integral decay estimate, where the extra term in \eqref{eq:intosc0} appears unavoidably and can not be removed unlike similar integral estimates obtained previously 
in the Euclidean setting in \cite{Duz-Min, Lieb--gen}, see Remark \ref{rem:noneuc}. Hence, one gets weaker integral decay estimate of the oscillation of the gradient of solutions of the 
inhomogeneous solution. Nevertheless, a perturbation lemma, Lemma \ref{lem:camp}, similar to the standard lemma of Campanato \cite{Camp, Gia}, leads to the $C^{1,\alpha}$-regularity of weak solutions of equation \eqref{eq:maineq}, exploiting the high integrability of the data. 

We develop necessary notations, definitions and provide previous results on sub-elliptic equations in Section \ref{sec:Prelim}. Then we prove the intermediate estimates in Section \ref{sec:est} and finally, we prove Theorem \ref{thm:c1au} in Section \ref{sec:proofthms}.

\section{Preliminaries and Previous results}\label{sec:Prelim}
\subsection{The Heisenberg Group}\label{subsec:Heisenberg Group}
Here we provide the definition and properties of Heisenberg group  
that would be useful in this paper.  
For more details, we 
refer to \cite{Bonfig-Lanco-Ugu, C-D-S-T}, etc. 
The \textit{Heisenberg Group}, denoted by $\hh^n$ for $n\geq 1$, is identified to  the Euclidean space 
$\R^{2n+1}$ with the group operation 
\begin{equation}\label{eq:group op}
 x\circ y\, := \Big(x_1+y_1,\ \dots,\ x_{2n}+y_{2n},\ t+s+\frac{1}{2}
\sum_{i=1}^n (x_iy_{n+i}-x_{n+i}y_i)\Big)
\end{equation}
for every $x=(x_1,\ldots,x_{2n},t),\, y=(y_1,\ldots,y_{2n},s)\in {\mathbb H}^n$.
Thus, $\hh^n$ with $\circ$ of \eqref{eq:group op} forms a 
non-Abelian Lie group, whose left invariant vector fields corresponding to the
canonical basis of the Lie algebra, are
\[ X_i=  \partial_{x_i}-\frac{x_{n+i}}{2}\partial_t, \quad
X_{n+i}=  \partial_{x_{n+i}}+\frac{x_i}{2}\partial_t,\] 
for every $1\leq i\leq n$ and the only
non zero commutator $ T= \partial_t$. 
We have 
\begin{equation}\label{eq:comm}
  [X_i\,,X_{n+i}]=  T\quad 
  \text{and}\quad [X_i\,,X_{j}] = 0\ \ \forall\ j\neq n+i, 
\end{equation}
and we call $X_1,\ldots, X_{2n}$ as \textit{horizontal
vector fields} and $T$ as the \textit{vertical vector field}. 

Given any scalar function $ f: \hh^n \to \R$, we denote
$\X f  = (X_1f,\ldots, X_{2n}f)$ the \textit{horizontal gradient} 
and $\XX f =  (X_i(X_j f))_{i,j} $
as the \textit{horizontal Hessian}. Also, 
the \textit{sub-Laplacian} operator is denoted by 
$\laph f = \sum_{j=1}^{2n}X_jX_jf $.   
For a vector valued function 
$F = (f_1,\ldots,f_{2n}) : \hh^n\to \R^{2n}$, the 
\textit{horizontal divergence} is defined as 
$$ \dvh (F)  =  \sum_{i=1}^{2n} X_i f_i .$$
%We notify the reader that 
%for notational convenience, we shall sometimes use the vector fields $X_j$ 
%acting on vector valued function component wise, with the notation 
%\begin{equation}\label{eq:notation}
%X_j(F)  :=  (X_jf_1,\ldots,X_jf_{2n}) 
%\end{equation}
%for all $j\in\{1,\ldots, 2n\}$. 
The Euclidean gradient of a scalar 
function $g: \R^{k} \to \R$, shall be denoted by
$\gr g=(D_1g,\ldots,D_{k} g)$ and the Hessian matrix by $D^2g$.

%A piecewise smooth rectifiable curve $\gamma $ is called a \textit{horizontal curve} if its tangent vectors are contained in the \textit{horizontal sub-bundle} 
%$\h = \spn\{X_1,\ldots,X_{2n}\}$ , i.e. 
%$\gamma'(t) \in \h_{\gamma(t)}$ for almost every $t$. 
%For any 
%$x, y \in \hh^n$, if the set of all horizontal curves is denoted as  
%$$\Gamma(x,y) = \set{\gamma:[0,1]\to \hh^n}{\gamma(0) = x,\gamma(1) = y, 
%\ \gamma'(t)\in\h_{\gamma(t)}},$$  
%then Chow's accessibility theorem (see \cite{Chow}) gurantees $\Gamma(x,y) \neq \emp $. 
The \textit{Carnot-Carath\`eodory metric} (CC-metric) is defined 
as the length 
%$\ell(\gamma)$ 
of the shortest 
horizontal curves connecting two points, see \cite{C-D-S-T}, 
and is denoted by $d$.  
%\begin{equation}\label{eq:cc metric}
%d(x,y)= \ \inf\set{\ell(\gamma)}{\gamma \in \Gamma(x,y)}. 
%\end{equation}
This is equivalent to the homogeneous metric, denoted as  
$ \dhh(x,y)= \normh{y^{-1}\circ x}$,  
where the homogeneous norm for $x=(x_1,\ldots,x_{2n}, t)\in \hh^n$ is 
\begin{equation}\label{eq:norm}
 \normh{x} :=  \Big(\, \sum_{i=1}^{2n} x_i^2+ |t|\, \Big)^\frac{1}{2}.
\end{equation}
Throughout this article we use the CC-metric balls $B_r(x) = \set{y\in\hh^n}{d(x,y)<r}$ for $r>0$ and $ x \in \hh^n $. However, by virtue of the equivalence 
of the metrics, all assertions for CC-balls can be restated to any homogeneous metric balls. 

The Haar measure of $\hh^n$ is just the Lebesgue 
measure of $\R^{2n+1}$. For a measurable set $E\subset \hh^n$, we denote 
the Lebesgue measure as $|E|$. For an integrable function $f$, we denote 
$$ (f)_E = \intav_E f\dx = \frac{1}{|E|} \int_E f\dx .$$
The Hausdorff dimension with respect to the metric $d$ is also the homogeneous dimension of the group $\mathbb H^n$, which shall be denoted as 
%\begin{equation}\label{eq:dimh}
%Q := 2n+2
%\end{equation}
$Q=2n+2$, 
throughout this paper. Thus, for any CC-metric ball $B_r$, we have that 
$|B_r| = c(n)r^Q$. 

For $ 1\leq p <\infty$, the \textit{Horizontal Sobolev space} $HW^{1,p}(\Omega)$ consists
of functions $u\in L^p(\Omega)$ such that the distributional horizontal gradient $\X u$ is in $L^p(\Omega\,,\R^{2n})$.
$HW^{1,p}(\Omega)$ is a Banach space with respect to the norm
\begin{equation}\label{eq:sob norm}
  \| u\|_{HW^{1,p}(\Omega)}= \ \| u\|_{L^p(\Omega)}+\| \X u\|_{L^p(\Omega,\R^{2n})}.
\end{equation}
We define $HW^{1,p}_{\loc}(\Omega)$ as its local variant and 
$HW^{1,p}_0(\Omega)$ as the closure of $C^\infty_0(\Omega)$ in 
$HW^{1,p}(\Omega)$ with respect to the norm in \eqref{eq:sob norm}. 
%This is a special case of Sobolev spaces defined on a general 
%metric measure spaces (see \cite{Haj-Kos}), where Sobolev-Poincar\'e inequalities and potential theory have 
%been developed. 
The Sobolev Embedding theorem has the following version in the setting of 
Heisenberg group, see \cite{Jerison,C-D-G,C-D-S-T} etc.
\begin{Thm}[Sobolev Inequality]\label{thm:sob emb}
Given $B_r\subset {\mathbb H}^n$ and $1<q<Q$, there exists $c=c(n,q)>0$ such that, for every $u \in HW^{1,q}_0(B_r)$ 
we have
\begin{equation}\label{eq:sob emb}
\left(\int_{B_r}| u|^{\frac{Q q}{Q-q}}\, dx\right)^{\frac{Q-q}{Q q}}
\leq\, c \left(\int_{B_r}| \X u|^q\, dx\right)^{\frac 1 q}.
\end{equation}
\end{Thm}

H\"older spaces with respect to homogeneous metrics have been defined in Folland-Stein \cite{Folland-Stein--book} and therefore, are sometimes known as Folland-Stein classes and denoted by $\Gamma^{\alpha}$ or 
$\Gamma^{\,0,\alpha}$ in some literature. However, as in \cite{Zhong, Muk-Zhong}, 
here we continue to maintain the classical notation and define 
\begin{equation}\label{def:holderspace}
C^{\,0,\alpha}(\Om) =
\set{u\in L^\infty(\Om)}{|u(x)-u(y)|\leq c\,d(x,y)^\alpha\ \forall\ x,y\in \Om}
\end{equation} 
for $0<\alpha \leq 1$, 
which are Banach spaces with the norm 
\begin{equation}\label{eq:holder norm}
\|u\|_{C^{\,0,\alpha}(\Om)}
=  \|u\|_{L^\infty(\Om)}+ \sup_{x,y\in\Om} 
\frac{|u(x)-u(y)|}{d(x,y)^\alpha}.
\end{equation}
These have standard extensions to classes $C^{k,\alpha}(\Om)$ 
for $k\in \N$, comprising functions having horizontal 
derivatives up to order $k$ in $C^{\,0,\alpha}(\Om)$; their local 
counterparts are denoted as $C^{k,\alpha}_\loc(\Om)$. The 
Morrey embedding theorem is the following.
\begin{Thm}[Morrey Inequality]\label{thm:mor emb}
Given any $B_r\subset {\mathbb H}^n$ and $q>Q$, there exists $c=c(n,q)>0$ such that, for every $u \in HW^{1,q}_0(B_r)\cap C(\bar B_r)$  we have
\begin{equation}\label{eq:mor emb}
|u(x)-u(y)|
\leq\, c\,d(x,y)^{1-Q/q} \left(\int_{B_r}| \X u|^q\, dx\right)^{\frac 1 q}, 
\quad\forall\ x,y\in B_r.
\end{equation}
\end{Thm}
%Now, 
%the definition of Morrey and Campanato spaces in sub-elliptic setting 
%differs in different texts. Here, we adopt the definition similar to 
%the classical one.  
%
%For any domain $\Om\subset \hh^n$ and $\lambda>0$, we define the {\it Morrey space} as 
%\begin{equation}\label{def:morrey}
%\mathcal M^{1,\lambda}(\Om) = 
%\set{u\in L^1_\loc(\Om)}{\int_{B_r} |u|\dx <c\,r^\lambda\ \forall\ B_r\subset\Om, r>0}
%\end{equation}
%and the  {\it Campanato space} as 
%\begin{equation}\label{def:camp}
%\mathcal L^{1,\lambda}(\Om) = 
%\set{u\in L^1_\loc(\Om)}{\int_{B_r} \big|u-\{u\}_{B_r}\big|\dx <c\,r^{\lambda}\ \forall\ B_r\subset\Om,r>0},
%\end{equation}
%where in both definitions $B_r$ represents balls with metric $d$. These spaces are Banach spaces and have properties similar to the classical spaces in the Euclidean setting. We shall use the fact that for every $0<\alpha<1$ and 
%$Q=2n+2$, we have  
%\begin{equation}\label{eq:holdcamp}
%\mathcal L^{1,Q+\alpha}(\Om) \subset C^{\,0,\alpha}(\Om),
%\end{equation}
%where the inclusion is to be understood as taking continuous representatives. 
%For details on classical Morrey and Campanato spaces, we refer to 
%\cite{Kuf-O-F} and for the sub-elliptic setting we refer to 
%\cite{C-D-S-T}. 

\subsection{Sub-elliptic equations} 
Here, we enlist some of the properties and results previously known for sub-elliptic equations 
of the form \eqref{eq:maineq}. 

First, we recall that the structure condition \eqref{eq:str} implies the monotonicity and ellipticity inequalities, as follows: 
\begin{align}\label{eq:monotone}
\inp{a(x,z_1)-a(x,z_2)}{z_1-z_2}&\geq c(|z_1|^2+|z_2|^2+s^2)^\frac{p-2}{2}
|z_1-z_2|^2\\
\label{eq:ellip}
\inp{a(x,z)}{z}&\geq c(|z|^2+s^2)^\frac{p-2}{2}|z|^2 
\end{align}
for some $c= c(n,p,L)>0$. This ensures existence and local uniqueness of 
weak solution $u\in HW^{1,p}(\Om)$ of equation \eqref{eq:maineq} from the 
classical theory of monotone operators, see \cite{Kind-Stamp}. We denote  
$u$ as the precise representative, hereafter. 

The regularity and apriori estimates of the homogeneous equation corresponding to \eqref{eq:maineq} with 
freezing of the coefficients, is necessary. Therefore,  
for any $x_0\in\Om$, we consider the equation 
\begin{equation}\label{eq:frzeq}
\dvh a(x_0, \X u) = 0 \quad\text{in}\ \Om. 
\end{equation}
%We recall the following zero-order potential estimate due to Trudinger-Wang \cite[Theorem 5.1]{Tru-Wang1}, see also \cite{Dan--bound}, which is the sub-elliptic analogue of the classical Wolff potential estimate of Kilpel\"ainen-Mal\'{y} \cite{Kilp-Mal1}. 
%\begin{Thm}\label{thm:truwang}
%If $u$ is a weak solution of the equation \eqref{eq:frzeq} with $a(x_0,z)$ satisfying the condition \eqref{eq:str}, then there exist $c=c(n,p,L)>0$ and 
%${\bar R}={\bar R}(n,p,L)>0$ such that 
%the estimate 
%\begin{equation}\label{eq:truwang}
%|u(x_0)|\,\leq\,  c\,|\sup_{\del B_R} u| + c\,\W{1}{p}(x_0,R) 
%\end{equation}
%holds for any $x_0\in\Om$ and $0<R\leq {\bar R}$, 
%whenever $B_{R}(x_0)\subset\Om$. 
%\end{Thm}

The $C^{1,\alpha}$ regulaity of $p$-Laplacian type equations has been dealt with in \cite{Zhong,Muk-Zhong}, where the equation 
$\dvh (D f(\X u))=0$ has been considered.
Given $D_z a(x_0,z)$ being symmetric, all the arguments there also follow in the same way for \eqref{eq:frzeq} with
the growth conditions \eqref{eq:str} which is the same as that in \cite{Muk-Zhong} and slightly weaker than that in \cite{Zhong} (in fact, \eqref{eq:frzeq} has been considered in \cite{Muk0} in a more general setting).  
The following regularity theorem is due to \cite[Theorem 1.1]{Zhong} and \cite[Theorem 1.3]{Muk-Zhong}. 

\begin{Thm}\label{thm:c1alpha}
If $u\in HW^{1,p}(\Om)$ is a weak solution of the equation \eqref{eq:frzeq} with $a(x_0,z)$ satisfying the condition \eqref{eq:str} and $D_z a(x_0,z)$ is a symmetric matrix, then $\X u$ is locally H\"older continuous. Moreover, there exist constants $c=c(n,p,L)>0$ and $\beta=\beta(n,p,L)\in (0,1)$ 
such that the following holds, 
\begin{align}
\label{eq:lipest}
(i)\ \ &\sup_{B_{R/2}}\ |\X u|^p\leq 
%\frac{c}{(1-\tau)^{Q/p}}
%c(1-\tau)^{-Q}
c\intav_{B_R}(|\X u|^2+s^2)^\frac{p}{2}\dx;\\
\label{eq:gradest}
(ii)\ \ &\intav_{B_\varrho}|\X u - (\X u)_{B_\varrho}|^p\dx \leq 
c\big(\varrho/R\big)^\beta 
\intav_{B_{R}}(|\X u|^2+s^2)^\frac{p}{2}\dx,
\end{align}
for every concentric $B_\varrho\subseteq B_R\subset\Om$ and $1<p<\infty$. 
\end{Thm}
%
%In addition, the sub-elliptic reverse H\"older inequality, see \cite{ZGold, Min-Z-Zhong}, and Gehring's lemma, implies that there exists 
%$\chi_0=\chi_0(n,p,L)>1$ such that we have 
%\begin{equation}\label{eq:subgehring}
%\Big(\intav_{B_{r/2}}(|\X u| +s)^{\chi_0 p}\dx\Big)^\frac{1}{\chi_0}
%\leq c\intav_{B_{r}}(|\X u| +s)^p\dx
%\end{equation}
%which, together with \eqref{eq:lipest}, yields
In fact, similarly as the Euclidean case, the following local estimate can be shown by using Sobolev's inequality and
Moser's iteration on the Caccioppoli type inequalities of \cite{Zhong}, for any $0<\sigma<1$ and $q>0$, 
\begin{equation}\label{eq:lipq}
\sup_{B_{\sigma R}}\ |\X u| \leq c (1-\sigma)^{-\frac{Q}{q}} \bigg(\intav_{B_R}(|\X u|^2+s^2)^\frac{q}{2}\dx\bigg)^\frac{1}{q}
\end{equation}
for some $c=c(n,p,L,q)>0$, see \cite[p. 12]{Zhong}. Thus, taking $q=1$, we can have 
\begin{equation}\label{eq:lip1}
\sup_{B_{R/2}}\ |\X u| \leq c\intav_{B_{R}}(|\X u| +s)\dx.
\end{equation}
From \eqref{eq:lip1} it ie easy to see that for all $0<r\leq R/2 $, we have 
\begin{equation}\label{eq:unlip}
\int_{B_r} |\X u|\dx\, \leq\, c\left(\frac{r}{R}\right)^Q\int_{B_{R}}(|\X u| +s)\dx,
\end{equation}
where $u\in C^{1,\beta}(\Om)$ is a solution of the equation \eqref{eq:frzeq} in the above inequalities. 

We recall the notion of De Giorgi's class of functions in this setting, which would be required for Proposition \ref{prop:intosc0}, in Section \ref{sec:est}. 
Given a metric ball $B_{\rho_0}\subset {\mathbb H}^n$, the De Giorgi's
class $DG^+(B_{\rho_0})$ consists of functions $v\in
HW^{1,2}(B_{\rho_0})\cap L^\infty(B_{\rho_0})$, which satisfy the inequality 
\begin{equation}\label{eq:DG}
\int_{B_{\rho'}}|\X (v-k)^+|^2\, dx\le
\frac{\gamma}{(\rho-\rho')^2}
\int_{B_\rho}|(v-k)^+|^2\, dx+\chi^2|
A_{k,\rho}^+|^{1-\frac{2}{Q}+\epsilon}
\end{equation}
for some $\gamma, \chi,\epsilon>0$, 
where $ A_{k,\rho}^+=\{ x\in B_\rho: (v-k)^+=\max (v-k,0)>0\}$ for any arbitrary $k\in\R$, the balls $B_{\rho'},B_\rho$ and 
$B_{\rho_0}$ are concentric with $0<\rho'<\rho\le \rho_0$. 
The class $DG^-(B_{\rho_0})$ is similarly defined and $DG(B_{\rho_0})=
DG^+(B_{\rho_0})\cap DG^-(B_{\rho_0})$. All properties of classical 
De Giorgi class functions, also hold for these classes. 

We end this section by introducing the sub-elliptic Wolff potential given by 
\begin{equation}\label{eq:Wolfdef}
\W{\beta}{p}(x_0,R) := \int^{R}_0 \Big( \frac{|\mu|(B_\varrho(x_0))}{\varrho^{Q-\beta p}}\Big)^\frac{1}{p-1}\frac{d\varrho}{\varrho}\qquad \forall\ \beta\in (0,Q/p],
\end{equation}
and 
recalling following lemma
of the density of Wolff potential, see \cite{Duz-Min} for proof.
\begin{Lem}\label{lem:denswolf}
Given any $H>1, x_0\in\Om$ and $r>0$, if $r_i= r/H^i$ for every $i\in \{0,1,2,\ldots\}$, then we have 
\begin{equation}\label{eq:denswolf}
\sum_{i=0}^{\infty} \bigg(\frac{|\mu|(B_{r_i}(x_0))}{r_{i}^{Q-1}}\bigg)^\frac{1}{p-1} \leq  \bigg(\frac{2^\frac{Q-1}{p-1}}{\log(2)}+\frac{H^\frac{Q-1}{p-1}}{\log(H)}\bigg)\W{\frac{1}{p}}{p}(x_0,2r).
\end{equation}
\end{Lem}

%Similarly as the classical case, the sub-elliptic Wolff potential is defined as 
%\begin{equation}\label{eq:Wolff}
%\W{\beta}{p}(x_0,R) := \int^{\bar R} \Big( \frac{|\mu|(B_\varrho(x_0))}{\varrho^{Q-\beta p}}\Big)^\frac{1}{p-1}\frac{d\varrho}{\varrho}\qquad \forall\ \beta\in (0,Q/p].
%\end{equation}

\section{Estimates of the horizontal gradient}\label{sec:est}
In this section, we show several comparison estimates along the lines of 
\cite{Kilp-Mal1, Duz-Min} ultimately leading to a pointwise estimate of the horizontal gradient. Here onwards we fix $x_0\in\Om$ and 
denote $B_\varrho = B_\varrho(x_0) $ for every $\varrho>0$. Also, we 
denote all constants as $c$, the values of which may vary from line to line but they are positive and dependent only on $n,p,L$, unless explicitly specified otherwise.

In the following, first we show the integral oscillation decay estimate of solutions of the equation \eqref{eq:frzeq}, analogous to that of the Euclidean setting in \cite{Lieb--gen, Duz-Min} etc. 
\begin{Prop}\label{prop:intosc0}
Let $B_{r_0}\subset \Om$ and $u\in C^{1,\beta}(\Om)$ be a solution of equation
\eqref{eq:frzeq}, with $\beta=\beta(n,p,L)\in (0,1)$. Then there exists 
$c=c(n,p,L)>0$, 
such that for all $0<\varrho<r<r_0$, we have 
\begin{equation}\label{eq:intosc0}
\intav_{B_\varrho}|\X u - (\X u)_{B_\varrho}|\dx \leq c\left(\frac{\varrho}{r}\right)^\beta
\Big[\intav_{B_r}|\X u - (\X u)_{B_r}|\dx + \chi r^\beta\Big]
\end{equation}
with $\chi= (s+M(r_0))/r_0^\beta$, where $M(r_0)=\max_{1\le i\le 2n}\sup_{B_{r_0}} |X_i u|$.
\end{Prop}
\begin{proof}
Given $B_{r_0}\subset \Om$, let us denote $M(\rho)=\max_{1\le i\le 2n}\sup_{B_{\rho}} |X_i u|$ and  
\begin{equation}\label{eq:Mom}
\omega(\rho)=\max_{1\le i\le 2n}\osc_{B_{\rho}} X_i u \quad\text{and}\quad 
I(\rho) = \intav_{B_\rho}|\X u - (\X u)_{B_\rho}|\dx
\end{equation}
for every $0<\rho<r_0$. 
Hence, note that $\omega(\rho)\leq 2M(\rho)$. Now, we recall 
the oscillation lemma proved in \cite[Theorem 4.1]{Muk-Zhong}, that 
there exists $m=m(n,p,L)\ge 0$ such that
for every $0<r\leq r_0/16$, we have  
\begin{equation}\label{mur4r'}
\omega(r) \leq (1-2^{-m})\omega(8r) + 2^m (s+M({r_0}))\Big(\frac{r}{{r_0}}\Big)^\beta,
\end{equation}
for some $\beta=\beta(n,p,L)\in (0,1/p)$. A standard iteration on \eqref{mur4r'}, 
see for instance \cite[Lemma 7.3]{Giu}, implies that for every $0<\varrho<r\leq r_0$, 
we have 
\begin{equation}\label{oscom}
\begin{aligned}
\omega(\varrho)\leq c\Big[\left(\frac{\varrho}{r}\right)^\beta
\omega(r)+\chi \varrho^\beta\Big]=c\left(\frac{\varrho}{r}\right)^\beta
\big[\omega(r)+\chi r^\beta\big]
\end{aligned}
\end{equation}
where $\chi= (s+M(r_0))/r_0^\beta$ and $c=c(n,p,L)>0$.  
If $\varrho \leq \delta r$ for some $\delta\in (0,1)$, 
it is easy to see from \eqref{oscom}, that for some $c=c(n,p,L)>0$, we have 
\begin{equation}\label{I<om}
I(\varrho) \leq c\,\omega(\varrho) \leq 
c\,\delta^{-\beta}\left(\frac{\varrho}{r}\right)^\beta[\omega(\delta r)+\chi r^\beta].
\end{equation}
Now we claim that, 
there exists $\delta=\delta(n,p,L)\in (0,1)$ such that, the inequality  
\begin{equation}\label{claimdel}
\omega(\delta r) \leq c [I(r) +\chi r^\beta]
\end{equation}
holds for some $c=c(n,p,L)>0$. Then \eqref{I<om} and \eqref{claimdel} together, 
yields \eqref{eq:intosc0}; hence proving the claim 
\eqref{claimdel} is enough to complete the proof. 

To this end, let us denote $r'=\delta r$, where $\delta\in(0,1)$ is to be chosen later. 
Notice that, to prove the claim \eqref{claimdel},  
we can make the apriori assumption:
\begin{equation}\label{assM}
\omega(r)\geq (s+M(r_0)) (r/r_0)^\sigma, 
\end{equation}
with $\sigma=1/p$ for $p\geq 2$, 
since, otherwise \eqref{claimdel} holds trivially with $\beta=\sigma$. Now, we consider the following complementary cases. This is very standard for elliptic estimates, see \cite{Dib, Tolk, Lieb--gen, Duz-Min} for corresponding Euclidean cases.

\textbf{Case 1}: For at least one index $l\in\{1,\ldots,2n\}$, we have either
\begin{equation*}
  \Big|B_{4r'}\cap\Big\{X_lu<\frac{M(4r')}{4}\Big\}\Big|
 \leq \theta |B_{4r'}| \ \ \text{or}\ \ \Big|B_{4r'}\cap\Big\{ X_lu>-\frac{M(4r')}{4}\Big\}\Big|
  \leq \theta |B_{4r'}|. 
 \end{equation*}
It has been shown in \cite[Theorem 4.1]{Muk-Zhong} that under assumption 
\eqref{assM}, if Case 1 holds with choice of a small enough $\theta = \theta(n,p,L)>0$,  
then $X_iu\in DG(B_{2r'})$ for every $i\in \{1,\ldots,2n\}$. 
Then, the standard local boundedness estimates of De Giorgi class functions 
\cite[Theorem 7.2 and 7.3]{Giu} follow; 
the fact that $X_iu$ belongs to $DG^+(B_{2r'})$ and $DG^-(B_{2r'})$,  
yields the following respective estimates for any $\vartheta < M(r')$:
%\begin{equation*}
%\sup_{B_{r'}}(X_iu-\vartheta) \leq c\Big[\intav_{B_{2r'}}(X_iu-\vartheta)^+\dx +\chi r^\beta\Big],\ 
%\sup_{B_{r'}}(\vartheta-X_iu) \leq c\Big[\intav_{B_{2r'}}(\vartheta-X_iu)^+\dx +\chi r^\beta\Big],
%\end{equation*}
\begin{align}
\label{supdg}
&\sup_{B_{r'}}(X_iu-\vartheta) \leq c\Big[\intav_{B_{2r'}}(X_iu-\vartheta)^+\dx +\chi r'^\beta\Big], \\
\label{infdg}
&\sup_{B_{r'}}(\vartheta-X_iu) \leq c\Big[\intav_{B_{2r'}}(\vartheta-X_iu)^+\dx +\chi r'^\beta\Big],
\end{align}
for every $i\in \{1,\ldots,2n\}$. Adding \eqref{supdg} and \eqref{infdg} with 
$\vartheta = (X_iu)_{B_{r'}}$, we get 
$$ \osc_{B_{r'}}X_iu\leq c\Big[\intav_{B_{2r'}}|X_iu-(X_iu)_{B_{r'}}|\dx +\chi r'^\beta\Big]\leq c[I(r)+\chi r^\beta]$$
for some $c=c(n,p,L)>0$ and $\delta<1/2$, 
which further implies \eqref{claimdel} for this case. 

\textbf{Case 2}: With $\theta=\theta(n,p,L)>0$ as in Case 1, for every $i\in\{1,\ldots,2n\}$, we have 
\begin{equation*}
  \Big|B_{4r'}\cap\Big\{X_lu<\frac{M(4r')}{4}\Big\}\Big|
 > \theta |B_{4r'}| \ \ \text{and}\ \ \Big|B_{4r'}\cap\Big\{ X_lu>-\frac{M(4r')}{4}\Big\}\Big|
  > \theta |B_{4r'}|. 
 \end{equation*}
First, we notice that the above assertions respectively imply $\inf_{B_{4r'}}X_iu\leq M(4r')/4$ 
and $\sup_{B_{4r'}}X_iu\geq -M(4r')/4$ for every 
$i\in\{1,\ldots,2n\}$. These further imply that
\begin{equation}\label{ocsup}
\omega(4r')\geq M(4r')-M(4r')/4=3M(4r')/4.
\end{equation}
Now, let us denote $L=\max_{1\leq i\leq 2n}|(X_iu)_{B_r}| = |(X_ku)_{B_r}|$ for some 
$k\in\{1,\ldots,2n\}$. Then note that, if $L>2\omega(4r')$ then using \eqref{ocsup},  we have
$$|(X_ku)_{B_r}|-|X_ku|\geq 2\omega(4r')-M(4r') \geq M(4r')/2 \quad\text{in} \  B_{4r'},$$
which, together with the choice of $\delta<1/4$, further implies
\begin{equation}\label{c21}
I(r)\geq c(n)\intav_{B_{4r'}}|X_k u-(X_ku)_{B_r}|\dx \geq 
\frac{c(n)}{2}M(4r')\geq \frac{c(n)}{4}\omega(4r'). 
\end{equation}
If $L\leq 2\omega(4r')=2\omega(4\delta r)$ then, 
we choose $\delta<1/8$ so that   
using $\omega(r/2)\leq 2M(r/2)$ and \eqref{eq:lip1} i.e.  
$M(r/2)\leq c\intav_{B_r}|\X u|\dx$ respectively on \eqref{oscom}, we obtain
\begin{equation}\label{llest1}
\begin{aligned}
\omega(4\delta r)&\leq c(8\delta)^\beta [\omega(r/2)+\chi r^\beta]\leq 
c\delta^\beta\Big[\intav_{B_r}|\X u|\dx+\chi r^\beta\Big] \\
&\leq c_1\delta^\beta[ I(r)+ L+\chi r^\beta]
\leq c_1\delta^\beta[ I(r)+ 2\omega(4\delta r)+\chi r^\beta]
\end{aligned}
\end{equation}
for some $c_1=c_1(n,p,L)>0$, 
where the second last inequality of the above is a consequence of triangle inequality and the 
definition of $I$ and $L$. Now we make a further reduction of $\delta$, such that 
$2c_1\delta^\beta<1$, so that \eqref{llest1} imply
\begin{equation}\label{c22}
\omega(4\delta r)\leq \frac{c_1\delta^\beta}{1-2c_1\delta^\beta}
\big[I(r)+\chi r^\beta\big].
\end{equation}
Thus \eqref{c21} and \eqref{c22} together shows that \eqref{claimdel} holds for Case 2, as well. Therefore, we have shown that claim \eqref{claimdel} holds for 
both cases and the proof is finished.
\end{proof}

\begin{Rem}\label{rem:noneuc}
For the Euclidean case, say $\dv(|\gr u|^{p-2}\gr u)=0$, it is well known, see \cite{Lieb--gen, Duz-Min}, that for any $0<\varrho<r$, the following estimate holds:
\begin{equation}\label{eq:eucstrong}
 \intav_{B_\varrho}|\gr u - (\gr u)_{B_\varrho}|\dx \leq c\left(\frac{\varrho}{r}\right)^\beta\intav_{B_r}|\gr u - (\gr u)_{B_r}|\dx .
\end{equation}
The purpose of the Proposition \ref{prop:intosc0} is to show that the sub-elliptic setting is very different even for the homogeneous equation and the integral oscillation estimate is not as strong as the above. We have the extra term $\chi\neq 0$ in \eqref{eq:intosc0} which one can not get rid of from estimates in \cite{Zhong,Muk-Zhong}. Its source goes back to the extra terms containing the commutator 
$Tu = [X_iu,X_{n+i}u]$
in the De Giorgi type estimates of \cite{Zhong,Muk-Zhong}, where $Tu$ is locally 
majorized by $\X u$ from an integrability estimate in   
\cite{Zhong}.
\end{Rem}

Thus, if $u\in C^{1,\beta}(\Om)$ is a solution of the equation \eqref{eq:frzeq}, the integral decay estimate of the oscillation we end up with from \eqref{eq:lip1} and \eqref{eq:intosc0}, is 
\begin{equation}\label{eq:intoscg}
\intav_{B_\varrho}|\X u - (\X u)_{B_\varrho}|\dx \leq c\left(\frac{\varrho}{r}\right)^\beta
\intav_{B_r}(|\X u|+s)\dx,
\end{equation}
for any $0<\varrho<r$, which is not as strong as the \eqref{eq:eucstrong}. Nevertheless, it is good enough for proving 
Theorem \ref{thm:c1au} via a perturbation argument.

\subsection{Comparison estimates} In this subsection, we prove comparison estimates essential for the proof of our theorems, by
localizing the equations \eqref{eq:maineq} and \eqref{eq:frzeq}. They follow similarly, mutatis mutandis, of the Euclidean case in \cite{Duz-Min}.
Here onwards, we denote $u\in HW^{1,p}(\Om)$ as a weak solution of \eqref{eq:maineq} and $p\geq 2$.

Fix $R>0$ such that $B_{2R}\subset \Om$ and consider the Dirichlet problem 
\begin{equation}\label{eq:homdir}
 \begin{cases}
  \dvh a(x,\X w)=  0\ \ \text{in}\ B_{2R};\\
 \ w-u\in HW^{1,p}_0(B_{2R}). 
 \end{cases}
\end{equation}
The following is the first comparison lemma where the density of the 
Wolff potential \eqref{eq:Wolfdef} appears in the estimates. The proof is similar to that of \cite{Duz-Min}, see also \cite{Boc-Gal}.
\begin{Lem}\label{lem:homcomp}
Let $u\in HW^{1,p}(\Om)$ be a weak solution of equation \eqref{eq:maineq} and $p\geq 2$. Then, the weak solution 
$w\in HW^{1,p}(B_{2R})$ of the equation \eqref{eq:homdir} satisfy 
\begin{equation}\label{eq:homdirest}
\intav_{B_{2R}}|\X w -\X u|\dx \leq c\bigg(\frac{|\mu|(B_{2R})}{R^{Q-1}}\bigg)^\frac{1}{p-1}, 
\end{equation}
for some $c=c(n,p,L)>0$. 
%\begin{align}\label{eq:homdir1}
%&(i)\qquad \intav_{B_{2R}}|w -u|\dx \leq c\bigg(\frac{|\mu|(B_{2R})}{R^{Q-p}}\bigg)^\frac{1}{p-1};\\
%\label{eq:homdir2}
%&(ii)\qquad \intav_{B_{2R}}|\X w -\X u|\dx \leq c\bigg(\frac{|\mu|(B_{2R})}{R^{Q-1}}\bigg)^\frac{1}{p-1}.
%\end{align}
\end{Lem}
\begin{proof}
By testing equation \eqref{eq:homdir} with  
$\varphi\in HW^{1,p}_0(B_{2R})$ and using equation \eqref{eq:maineq}, 
we have the weak formulation
\begin{equation}\label{eq:weakhomdir}
\int_{B_{2R}}\inp{a(x,\X u)-a(x,\X w)}{\X\varphi}\dx = \int_{B_{2R}}\varphi\, d\mu
\end{equation}
which we estimate with appropriate choices of $\varphi$, in order to show \eqref{eq:homdirest}.

First, we assume $2\leq p\leq Q$. For any $j\in\N$, we denote the following truncations 
\begin{equation*}
\psi_j = \max\bigg\{-\frac{j}{R^\gamma}, \min\Big\{\frac{u-w}{m},\frac{j}{R^\gamma}\Big\}\bigg\},\  
\varphi_j = \max\bigg\{-\frac{1}{R^\gamma}, \min\Big\{\frac{u-w}{m}-\psi_j,\frac{1}{R^\gamma}\Big\}\bigg\}, 
\end{equation*}
where the scaling constants $m,\gamma\geq 0$ are to be chosen later. Notice that, 
for each $j\in\N$, we have $|\varphi_j|\leq 1/R^\gamma$ and 
$\X\varphi_j = \frac{1}{m}(\X u-\X w)\1_{E_j}$ where 
$$E_j=\{m j/R^\gamma<|u-w|\leq m(j+1)/R^\gamma\}.$$ 
Taking $\varphi=\varphi_j$ in \eqref{eq:weakhomdir} and using \eqref{eq:monotone} with $p\geq 2$,
it is easy to obtain 
\begin{equation}\label{eq:cj}
\int_{B_{2R}\cap E_j}
|\X w -\X u|^p\dx \,\leq\, \frac{cm}{R^\gamma}|\mu|(B_{2R})
\end{equation}
for some $c=c(n,p,L)>0$. Now, using H\"older's inequality and \eqref{eq:cj}, 
we obtain
$$\int_{B_{2R}\cap E_j}|\X w -\X u|\dx \leq |E_j|^\frac{p-1}{p}\Big(\int_{B_{2R}\cap E_j}
|\X w -\X u|^p\dx\Big)^\frac{1}{p} \leq  c|E_j|^\frac{p-1}{p}(m /R^\gamma)^\frac{1}{p}|\mu|(B_{2R})^\frac{1}{p},$$
then, using the fact that $|u-w|^\kappa>(m j/R^\gamma)^\kappa$ in $E_j$, we obtain 
\begin{equation}\label{eq:1cj}
\begin{aligned}
\int_{B_{2R}\cap E_j}|\X w -\X u|\dx 
\leq \frac{c(m /R^\gamma)^\frac{1}{p}|\mu|(B_{2R})^\frac{1}{p}}{(m j/R^\gamma)^\frac{\kappa(p-1)}{p}}\bigg(\int_{B_{2R}\cap E_j} |u-w|^\kappa\dx\bigg)^\frac{p-1}{p}
\end{aligned}
\end{equation}
with $\kappa=Q/(Q-1)$. Also from \eqref{eq:cj}, note that for any $N\in\N$,   
\begin{equation}\label{eq:dj}
\begin{aligned}
\int_{B_{2R}\cap\{|u-w|\leq m N/R^\gamma\}} |\X w -\X u|^p\dx
&= \sum_{j=0}^{N-1}
\int_{B_{2R}\cap E_j}|\X w -\X u|^p\dx \\
&\leq \frac{c\,m}{R^\gamma}N|\mu|(B_{2R}).
\end{aligned}
\end{equation}
Now, we estimate the whole 
integral using \eqref{eq:dj} and \eqref{eq:1cj}, as follows:
\begin{align*}
\int_{B_{2R}}&|\X w -\X u|\dx \\
&= \int_{B_{2R}\cap\{|u-w|\leq m N/R^\gamma\}}
|\X w -\X u|\dx +
\int_{B_{2R}\cap\{|u-w|> m N/R^\gamma\}}|\X w -\X u|\dx\\
&\leq |B_{2R}|^\frac{p-1}{p}\bigg(\int_{B_{2R}\cap\{|u-w|\leq m N/R^\gamma\}} |\X w -\X u|^p\dx\bigg)^\frac{1}{p}
+ \sum_{j=N}^{\infty}
\int_{B_{2R}\cap E_j}|\X w -\X u|\dx \\
&\leq c(m /R^\gamma)^\frac{1}{p}|\mu|(B_{2R})^\frac{1}{p}
\bigg(|B_{2R}|^\frac{p-1}{p}N^\frac{1}{p} 
+ \sum_{j=N}^{\infty}\Big[\frac{1}{(m j/R^\gamma)^\kappa}\int_{B_{2R}\cap E_j} |u-w|^\kappa\dx\Big]^\frac{p-1}{p}\bigg)
\end{align*}
Using Sobolev inequality \eqref{eq:sob emb} on the second term of the above, we obtain
\begin{align*}
\int_{B_{2R}}|\X w -\X u|\dx &\leq c(m /R^\gamma)^\frac{1}{p}|\mu|(B_{2R})^\frac{1}{p}
|B_{2R}|^\frac{p-1}{p}N^\frac{1}{p} \\
&\ + c(m /R^\gamma)^{\frac{1}{p}-\frac{\kappa(p-1)}{p}}|\mu|(B_{2R})^\frac{1}{p}\,\epsilon(N)^\frac{1}{p}\Big(\int_{B_{2R}} 
|\X u-\X w|\dx\Big)^\frac{\kappa(p-1)}{p}
\end{align*}
where $\epsilon(N)=\sum_{j=N}^{\infty}1/j^{\kappa(p-1)},\,\kappa=Q/(Q-1)$ and $c=c(n,p,L)>0$. 

Now, first we consider the case $p<Q$, so that we have $\kappa(p-1)/p<1$. Then, by applying Young's inequality on the second term, we obtain
\begin{equation*}
\begin{aligned}
\int_{B_{2R}}|\X w -\X u|\dx &\leq c\Big(\frac{m}{R^\gamma}\Big)^\frac{1}{p}|\mu|(B_{2R})^\frac{1}{p}
|B_{2R}|^\frac{p-1}{p}N^\frac{1}{p}
+c\Big(\frac{m}{R^\gamma}\Big)^\frac{1+\kappa-\kappa p}{p+\kappa-\kappa p}|\mu|(B_{2R})^\frac{1}{p+\kappa-\kappa p}\\
 &\qquad + \epsilon(N)^\frac{1}{\kappa(p-1)}\Big(\int_{B_{2R}} 
|\X u-\X w|\dx\Big)
\end{aligned}
\end{equation*}
for some $c=c(n,p,L)>0$. 
Now, we make the following choice of the scaling constants, 
$$ m = |\mu|(B_{2R})^\frac{1}{p-1}\quad \text{and}\quad 
\gamma= (Q-p)/(p-1)$$ 
such that the first two terms of the above are the same. Also note that, 
since $p\geq 2>1+1/\kappa$, we have 
$\kappa(p-1)>1$ and hence, $\sum_{j=1}^{\infty}1/j^{\kappa(p-1)} = 
\zeta(\kappa(p-1))<\infty$. If $N$ is large enough, we can have 
 $$\epsilon(N)=\sum_{j=N}^{\infty}1/j^{\kappa(p-1)}<1/2^{\kappa(p-1)}$$ and thus, the 
last term of the estimate can be absolved in the left hand side. With these choices of $m,\gamma,N$, we finally obtain
\begin{equation}\label{eq:homfin}
\int_{B_{2R}}|\X w -\X u|\dx \leq 
c |\mu|(B_{2R})^\frac{1}{p-1} R^{\frac{Qp-2Q+1}{p-1}}
\end{equation}
for some $c=c(n,p,L)>0$, which immediately implies \eqref{eq:homdirest}. 

For the case of $p=Q$, the estimate \eqref{eq:homfin} also follows similary with a possibly larger $N$ and the same choices of scaling constants, i.e. $m = |\mu|(B_{2R})^{1/(Q-1)}$ and $\gamma=0$; except here 
we absolve the last term to the right hand side directly, without using Young's inequality. 

Now we assume the $p\geq Q$. Here we simply choose $\varphi=u-w$ in \eqref{eq:weakhomdir} and use \eqref{eq:monotone} together with Morrey's 
inequality \eqref{eq:mor emb} to obtain 
\begin{align*}
\int_{B_{2R}}|\X w -\X u|^p\dx &\leq c\int_{B_{2R}}|u-w|\,d\mu 
\leq c|\mu|(B_{2R})\sup_{B_{2R}} |u-w| \\
&\leq c|\mu|(B_{2R}) R^{1-\frac{Q}{p}}
\Big(\int_{B_{2R}}|\X w -\X u|^p\dx\Big)^\frac{1}{p}, 
\end{align*}
which, upon using Young's inequality, yields
\begin{equation}\label{eq:mordir}
\int_{B_{2R}}|\X w -\X u|^p\dx 
\leq c |\mu|(B_{2R})^\frac{p}{p-1} R^\frac{p-Q}{p-1}.
\end{equation}
Then, using H\"older's inequality and \eqref{eq:mordir}, we obtain
\begin{align*}
\int_{B_{2R}}|\X w -\X u|\dx &\leq |B_{2R}|^\frac{p-1}{p} \Big(\int_{B_{2R}}|\X w -\X u|^p\dx\Big)^\frac{1}{p}\\
&\leq c |\mu|(B_{2R})^\frac{1}{p-1} R^{\frac{Qp-2Q+1}{p-1}}
\end{align*}
which, just as before, implies \eqref{eq:homdirest}. 
Thus, the proof is finished. 
\end{proof}
\begin{Rem}
It is evident that by using Sobolev or Morrey inequality 
\eqref{eq:sob emb}, \eqref{eq:mor emb} on \eqref{eq:homdirest}, we can obtain the estimate 
$$  
\intav_{B_{2R}}|w -u|\dx \leq c\bigg(\frac{|\mu|(B_{2R})}{R^{Q-p}}\bigg)^\frac{1}{p-1}$$
where $u$ and $w$ are the functions stated in Lemma \ref{lem:homcomp}. 
\end{Rem}
For the next comparison estimate, we require the Dirichlet problem with freezing of  the coefficients. Letting $w\in HW^{1,p}(B_{2R})$ as weak solution of \eqref{eq:homdir}, we consider 
\begin{equation}\label{eq:frzdir}
 \begin{cases}
  \dvh a(x_0,\X v)=  0\ \ \text{in}\ B_R;\\
 \ v-w\in HW^{1,p}_0(B_R). 
 \end{cases}
\end{equation}
\begin{Lem}\label{lem:frzcomp}
Given weak solution $w\in HW^{1,p}(B_{2R})$ of \eqref{eq:homdir}, if $v\in HW^{1,p}(B_R)$ is the weak solution 
of equation \eqref{eq:frzdir}, then there exists $c=c(n,p,L)>0$ such that 
\begin{equation}\label{eq:frzcomp}
\intav_{B_R}|\X v-\X w|^p\dx\leq cL'^2R^{2\alpha}\intav_{B_R}(|\X w|+s)^p\dx.
\end{equation}
\end{Lem}
\begin{proof}	
First, note that by testing equation \eqref{eq:frzdir} with $w-v$ and using 
the ellipticity \eqref{eq:ellip}, it is not difficult to show the following inequality, 
\begin{equation}\label{eq:enbd}
\int_{B_R} |\X v|^p\dx \leq c\int_{B_R} (|\X w|+s)^p\dx,
\end{equation}
for some $c= c(n,p,L)$; the proof is standard.  
Also, testing both equations \eqref{eq:homdir} and \eqref{eq:frzdir} with 
$w-v$, we have that 
$$ \int_{B_{2R}} \inp{a(x,\X w)}{\X w -\X v}\dx=0=\int_{B_R}\inp{a(x_0,\X v)}{\X w -\X v}\dx .$$ 
Using the above together with \eqref{eq:monotone} and \eqref{eq:str}, we obtain
\begin{align*}
c\int_{B_R}&(|\X w|^2+|\X v|^2+s^2)^\frac{p-2}{2}|\X w-\X v|^2\dx \\
&\leq \int_{B_R}\inp{a(x_0,\X w)-a(x_0,\X v)}{\X w -\X v}\dx\\
&= \int_{B_R}\inp{a(x_0,\X w)-a(x,\X w)}{\X w -\X v}\dx\\
&\leq c L'R^\alpha \int_{B_R}(|\X w|^2+|\X v|^2+s^2)^\frac{p-2}{2}|\X w|
|\X w-\X v|\dx
\end{align*}
Using Young's inequality on the last integral of the above, it is easy to get 
$$ \int_{B_R}(|\X w|^2+|\X v|^2+s^2)^\frac{p-2}{2}|\X w-\X v|^2\dx
\leq c (L'R^\alpha)^2 \int_{B_R} (|\X w|^2+|\X v|^2+s^2)^\frac{p}{2}\dx. $$
This, together with \eqref{eq:enbd}, is enough to prove \eqref{eq:frzcomp}. 
\end{proof}

Combining Lemma \ref{lem:homcomp} and Lemma \ref{lem:frzcomp}, we obtain the following comparison estimate of weak solution $u$ of \eqref{eq:maineq} and 
weak solution $v$ of \eqref{eq:frzdir}. 
%\cite{Lu}

\begin{Cor}\label{cor:comb}
Let $u\in HW^{1,p}(\Om)$ be a weak solution of equation \eqref{eq:maineq} and let $v\in HW^{1,p}(B_R)$ be the weak solution 
of equation \eqref{eq:frzdir}, where
$w\in HW^{1,p}(B_{2R})$ given in the problem \eqref{eq:frzdir} is the weak solution of the equation \eqref{eq:homdir}. Then there exists $c=c(n,p,L)>0$ such that 
\begin{equation*}
\intav_{B_R}|\X v-\X u|\dx\leq c\big(1+(L'R^\alpha)^\frac{2}{p}\big)
\bigg(\frac{|\mu|(B_{2R})}{R^{Q-1}}\bigg)^\frac{1}{p-1} + c(L'R^\alpha)^\frac{2}{p}\intav_{B_{2R}}(|\X u|+s)\dx.
\end{equation*}
\end{Cor}
\begin{proof}
First, notice that H\"older's inequality and \eqref{eq:frzcomp} imply
\begin{equation}\label{eq:frzcomp'}
\intav_{B_R}|\X v-\X w|\dx\leq c(L'R^\alpha)^\frac{2}{p}
\Big(\intav_{B_R}(|\X w|+s)^p\dx\Big)^\frac{1}{p}.
\end{equation}
Hence, using \eqref{eq:homdirest} and \eqref{eq:frzcomp'}, we obtain 
\begin{equation}\label{eq:s1}
\begin{aligned}
\intav_{B_R}|\X v-\X u|\dx&\leq \intav_{B_R}|\X w-\X u|\dx + \intav_{B_R}|\X v-\X w|\dx\\
&\leq c\bigg(\frac{|\mu|(B_{2R})}{R^{Q-1}}\bigg)^\frac{1}{p-1}
+c(L'R^\alpha)^\frac{2}{p}
\Big(\intav_{B_R}(|\X w|+s)^p\dx\Big)^\frac{1}{p}.
\end{aligned}
\end{equation}
We estimate the last integral 
using sub-elliptic reverse H\"older's inequality and Gehring's lemma, see \cite{ZGold}, to obtain   
\begin{equation}\label{eq:s2}
\begin{aligned}
\Big(\intav_{B_R}(|\X w|+s)^p\dx\Big)^\frac{1}{p} &\leq 
c\intav_{B_{2R}}(|\X w|+s)\dx \\
&\leq c\intav_{B_{2R}}(|\X u|+s)\dx + c\intav_{B_{2R}}|\X u-\X w|\dx\\
&\leq c\intav_{B_{2R}}(|\X u|+s)\dx + c\bigg(\frac{|\mu|(B_{2R})}{R^{Q-1}}\bigg)^\frac{1}{p-1}, 
\end{aligned}
\end{equation}
where the last inequality follows from \eqref{eq:homdirest}. 
Now it is easy to see that by combining \eqref{eq:s1} and \eqref{eq:s2}, 
the proof is finished.
\end{proof}

\section{Proof of the Theorem \ref{thm:c1au}}\label{sec:proofthms}
We shall prove Theorem \ref{thm:c1au} in this section. 
As before, here we maintain $u\in HW^{1,p}(\Om)$ as a weak solution of the equation \eqref{eq:maineq} and fix 
some arbitrary $x_0\in \Om$ and denote the metric balls $B_\rho=B_\rho(x_0)$ for every $\rho>0$. The comparison 
estimates of Section \ref{sec:est} shall lead to the necessary estimates for $u$. 

With respect to the given data, let us set 
\begin{equation}\label{eq:r0}
 \bar R=\bar R(n,p,L,L',\alpha,\dist(x_0,\del\Om))>0,
\end{equation}
which shall be chosen as small as required as we proceed, finally the minimum of every reductions of $\bar R$, is to be considered.
Let $\bar R\leq \min\{1, \frac{1}{2}\dist(x_0,\del\Om), L'^{-1/\alpha}\}$ to begin with, so that for any $R<\bar R$, we have 
$R, L'R^\alpha <1$ and $B_R\subset \Om$. 

The following lemma is a consequence of the uniform Lipschitz estimate \eqref{eq:unlip}.

\begin{Lem}\label{lem:amlip}	
For any $0<\rho\leq R \leq \bar R/2$, we have the estimate
\begin{equation}\label{eq:amlip}
\begin{aligned}
\int_{B_\rho}(|\X u|+s)\dx\,&\leq\, c\left(\frac{\rho}{R}\right)^{Q}\int_{B_R} (|\X u|+s)\dx 
\,+\, cR^{Q}\bigg(\frac{|\mu|(B_{2R})}{R^{Q-1}}\bigg)^\frac{1}{p-1}\\
&\quad + c(L'R^\alpha)^\frac{2}{p}\int_{B_{2R}}(|\X u|+s)\dx.
\end{aligned} 
\end{equation}
%for any $\kappa_0\leq\kappa<Q$, 

\end{Lem}

\begin{proof}
We denote comparison 
function $v$ as the weak solution of equation \eqref{eq:frzdir}, as before.  Then we write 
\begin{equation}\label{o1}
\int_{B_\rho}(|\X u|+s)\dx \leq \int_{B_\rho}(|\X v|+s)\dx + \int_{B_\rho}|\X u-\X v|\dx 
\end{equation}
The first term is estimated from 
\eqref{eq:unlip} as 
\begin{equation}\label{o2}	
\begin{aligned}
\int_{B_\rho}(|\X v|+s)\dx &\leq  c\left(\frac{\rho}{R}\right)^Q\int_{B_R}(|\X v|+s)\dx \\
&\leq c\left(\frac{\rho}{R}\right)^Q\int_{B_R}(|\X u|+s)\dx +c\left(\frac{\rho}{R}\right)^Q\int_{B_R}|\X v-\X u|\dx.
\end{aligned}
\end{equation}
The last terms of \eqref{o1} and \eqref{o2} are estimated by Corollary \ref{cor:comb} and we end up with \eqref{eq:amlip}. This  concludes the proof. 

\end{proof}

The following Lemma is similar to a well-known lemma of 
Campanato \cite{Camp},\cite{Gia-Giu--sharp}. The proof follows along the same lines as in
{\cite[Lemma 2.1]{Gia}}.
%provided in the Appendix. 
 \begin{Lem}\label{lem:camp}
Let $\phi :(0,\infty)\to [0,\infty)$ be a non-decreasing functions, $A> 1$ and $\epsilon\geq 0$ be fixed constants. 
Let  $\psi, \Psi :(0,\infty)\to [0,\infty)$ be functions such that $\sum_{j=0}^\infty \psi(t^j r) \leq \Psi(r)$ for any $0<t<t_0<1$. 
Given any $a>0$, suppose that 
\begin{equation}\label{camp1}
 \phi(\rho)  \leq A\left[\left(\frac{\rho}{r}\right)^a + \epsilon\right]\phi(r)  + r^a\psi(r)
\end{equation}
holds for any $ 0<\rho < r\leq R_0$,
 then there exists constants $\epsilon_0=\epsilon_0(A,a)>0$ and $c=c(A, a)>0$ such that if $ \epsilon \leq \epsilon_0$, then 
for all $0<\rho<r\leq R_0$, we have
\begin{equation}\label{camp2}
\phi(\rho)  \leq c\left[\left(\frac{\rho}{r}\right)^{a-\bar\eps}\phi(r) +
  \rho^{a-\bar\eps}r^{\bar\eps} \Psi(r)\right], 
\end{equation}
for any $0<\bar\eps<a$. 
\end{Lem}

\begin{proof}%[Proof of Lemma \ref{lem:camp}]
We fix $0<r\leq R_0$. 
Notice that, for any $0<t<1$, \eqref{camp1} implies 
$$ \phi(tr) \leq At^a\left(1+ \frac{\epsilon}{t^a}\right)\phi(r) + r^a\psi(r). $$ 
%Since, $b<a$, we choose $\gamma=\gamma(a,b)$ so that $b<\gamma<a$. 
We fix some $t<t_0$ and let $\epsilon_0 \leq t^a$, so that for every $\epsilon\leq\epsilon_0$, we have 
$$At^a(1+\epsilon/t^a) \leq 2At^a.$$ Using this on the above, we get
\begin{equation}\label{campit}
\phi(tr) \leq 2At^a \phi(r) + r^a \psi(r).
\end{equation}
Now, we iterate \eqref{campit} as follows; for any $k\in \N$,  
\begin{equation*}
\begin{aligned}
\phi(t^{k+1}r) &\leq 2At^a \phi(t^k r) + t^{ka}r^a \psi(t^kr)\\
&\leq (2At^a)^2 \phi(t^{k-1}r) + 2At^{ka}r^a\psi(t^{k-1}r) + t^{ka}r^a \psi(t^kr)\\
&\leq (2At^a)^3 \phi(t^{k-2}r) + (2A)^2 t^{ka}r^a\psi(t^{k-2}r) +2At^{ka}r^a\psi(t^{k-1}r) + t^{ka}r^a \psi(t^kr)\\
&\leq \ldots \leq (2At^a)^{k+1} \phi(r) + t^{ka}r^a\sum_{j=0}^k (2A)^j \psi(t^{k-j}r).
\end{aligned}
\end{equation*}
Since, $\sum_{j=0}^\infty \psi(t^j r) \leq \Psi(r)$ as given, we have 
$$ \phi(t^{k+1}r) \leq (2At^a)^{k+1}\phi (r) + (2At^a)^{k}r^a \Psi(r).$$
Now, given any $0<\bar\eps<a$, we can choose $t$ small enough such that $t^{\bar\eps}\leq \frac{1}{2A}$ and hence 
$2At^a\leq t^{a-\bar\eps}$. Then, we have 
\begin{equation}\label{cff}
\begin{aligned}
\phi(t^{k+1}r) \leq c \Big[ t^{(k+1)(a-\bar\eps)}\phi(r) + t^{k(a-\bar\eps)}r^a \Psi(r) \Big]
\end{aligned}
\end{equation}
for some $c=c(A,a)>0$. 
Now, given any $\rho<r$, we can choose $k\in \N$ such that, we have 
$ t^{k+1}r < \rho\leq t^k r$. Then, \eqref{cff} implies
$$ \phi(t\rho) \leq \phi(t^{k+1}r) 
\leq c\left(\frac{\rho}{r}\right)^{a-\bar\eps}\phi(r) + \frac{c}{t^{a-\bar\eps}} \rho^{a-\bar\eps} r^{\bar\eps} \Psi(r), $$ 
which, with a rescaling of $\rho$ by constants dependent on $A,a$, yields \eqref{camp2}. This completes the proof.
\end{proof}
Using the above Lemma together with Lemma \ref{lem:amlip}, we obtain an almost-Lipschitz estimate, as follows. 
\begin{Prop}\label{prop:Aml}
There exists $c=c(n,p,L)>0$ such that,
\begin{equation}\label{eq:Aml}
\int_{B_r}(|\X u|+s)\dx\leq c \left(\frac{r}{R}\right)^{Q-\bar\eps}
\bigg[\int_{B_R}(|\X u|+s)\dx + R^Q\W{\frac{1}{p}}{p}(x_0,R)\bigg]
\end{equation}
holds for any $0<\bar\eps <Q$ and $0<r\leq R\leq \bar R$.
\end{Prop}

\begin{proof}
First, let us fix $0<r\leq \bar R$ and denote 
$$ \phi(r)= \int_{B_r}(|\X u|+s)\dx \quad \text{and}\quad w_\mu(r)= \bigg(\frac{|\mu|(B_{r})}{r^{Q-1}}\bigg)^\frac{1}{p-1}.$$ 
We recall \eqref{eq:amlip} with appropriate scaling, to have 
$$\phi(\rho)\leq c\left(\frac{\rho}{r}\right)^{Q}\phi(r) + c\,r^Q w_\mu(r) + (L'r^\alpha)^\frac{2}{p}\phi(r) ,$$
for any $\rho\leq r$ and $c=c(n,p,L)>0$. We can apply Lemma \ref{lem:camp} on the above with $a=Q$ and using approriate reduction 
$ (L'{\bar R}^\alpha)^\frac{2}{p} \leq \epsilon_0(n,p,L)$. Recalling \eqref{eq:denswolf}, notice that $w_\mu$ satisfy the summability condition of Lemma \ref{lem:camp} and we obtain 
$$ \phi(r)  \leq c\left[\left(\frac{r}{R}\right)^{Q-\bar\eps}\phi(R) +
  r^{Q-\bar\eps}R^{\bar\eps} \W{\frac{1}{p}}{p}(x_0,R)\right], $$
for every $0<r\leq R\leq \bar R$, and hence we have \eqref{eq:Aml}. This completes the proof. 
\end{proof}

Now, we use the estimate \eqref{eq:intoscg} along with the above estimates to prove $C^{1,\gamma}$ regularity of $u$. We continue to assume $\bar R$ subject to reductions with dependence of data as in \eqref{eq:r0}. First, we have the following lemma.
\begin{Lem}\label{lem:intosc}
 There exist 
$\beta=\beta(n,p,L)\in(0,1)$ and $c=c(n,p,L)>0$ such that, for every 
$0<\varrho < R < \bar R/2$, the following estimate holds:
\begin{equation*}
\begin{aligned}
\intav_{B_\varrho}|\X u - (\X u)_{B_\varrho}|\dx &\,\leq\, c\Big(\frac{\varrho}{R}\Big)^\beta  \intav_{B_{ R}}(|\X u|+s)\dx\\
 &\ + c\Big(\frac{R}{\varrho}\Big)^Q\left[
\bigg(\frac{|\mu|(B_{2R})}{R^{Q-1}}\bigg)^\frac{1}{p-1} +(L'R^\alpha)^\frac{2}{p}\intav_{B_{2R}}(|\X u|+s)\dx\right].
\end{aligned}
\end{equation*}
\end{Lem}
\begin{proof}
We define the comparison 
functions $w$ and $v$ as weak solutions of equations \eqref{eq:homdir} and \eqref{eq:frzdir}, as before. Then we have 
\begin{equation}\label{eq:io1}
\begin{aligned}
\intav_{B_\varrho}|\X u - (\X u)_{B_\varrho}|\dx &\leq 
2\intav_{B_\varrho}|\X u - (\X v)_{B_\varrho}|\dx \\
&\leq 2\intav_{B_\varrho}|\X v - (\X v)_{B_\varrho}|\dx 
+2\intav_{B_\varrho}|\X u - \X v|\dx . 
\end{aligned}
\end{equation}
Now, we shall estimate both terms of the right hand side of \eqref{eq:io1} 
seperately. 

Using \eqref{eq:intoscg}, we estimate 
the first term of \eqref{eq:io1} as 
\begin{equation*}
\begin{aligned}
\intav_{B_\varrho}|\X v - (\X v)_{B_\varrho}|\dx &\leq c\Big(\frac{\varrho}{R}\Big)^\beta 
\intav_{B_{R}}(|\X v|+s)\dx \\
&\leq  c\Big(\frac{\varrho}{R}\Big)^\beta \intav_{B_{R}}(|\X u|+s)\dx  
+  c\Big(\frac{\varrho}{R}\Big)^\beta \intav_{B_{R}}|\X v-\X u|\dx 
\end{aligned}
\end{equation*}
 
The second term of \eqref{eq:io1} is estimated simply as 
\begin{equation*}
\intav_{B_\varrho}|\X u - \X v|\dx \leq 
c\Big(\frac{R}{\varrho}\Big)^Q\intav_{B_R}|\X u - \X v|\dx. 
\end{equation*}
Using the above estimates in \eqref{eq:io1}, together with  
Corollary \ref{cor:comb} to estimate the integral of $|\X u -\X v|$, 
the proof is finished.
\end{proof}
Now we are ready to prove Theorem \ref{thm:c1au}. An extra dependence on $q$ is assumed on $\bar R$, where $q>Q$ is as in the statement of Theorem \ref{thm:c1au}. 

\begin{proof}[Proof of Theorem \ref{thm:c1au}]
Let us assume the notation $w_\mu$ for the density of the Wolff potential, as used in the previous subsection. 
From Lemma \ref{lem:intosc}, we get
\begin{equation}\label{eq:hold1}
\begin{aligned}
\int_{B_\varrho}|\X u - (\X u)_{B_\varrho}|\dx \,\leq\, c\Big(\frac{\varrho}{r}\Big)^{Q+\beta} & \int_{B_{r}}(|\X u|+s)\dx
 + c\, r^Q w_\mu(r)\\
 &\qquad+c(L'r^\alpha)^\frac{2}{p}\int_{B_{r}}(|\X u|+s)\dx
\end{aligned}
\end{equation}
and from \eqref{eq:Aml} of Proposition \ref{prop:Aml}, we have, 
\begin{equation}\label{eq:Aml2}
\int_{B_r}(|\X u|+s)\dx\leq c\left(\frac{r}{R}\right)^{Q-\bar\eps}
\bigg[\int_{B_R}(|\X u|+s)\dx + R^Q\W{\frac{1}{p}}{p}(x_0,R)\bigg].
\end{equation}
We use \eqref{eq:Aml2} on \eqref{eq:hold1} to obtain the following estimate, 
\begin{equation}\label{hl1}
\begin{aligned}
\int_{B_\varrho}|\X u - (\X u)_{B_\varrho}|\dx \,&\leq\, c 
\Big(\frac{\varrho^{Q+\beta}R^{\bar\eps}}{r^{\beta+\bar\eps} R^{Q}}\Big)
\bigg[\int_{B_R}(|\X u|+s)\dx + R^Q\W{\frac{1}{p}}{p}(x_0,R)\bigg]\\
 &\quad  + c\, r^Q  \bigg[w_\mu(r)\, +  (L'r^\alpha)^\frac{2}{p}\intav_{B_{r}}(|\X u|+s)\dx\bigg] .
\end{aligned}
\end{equation}
for every $0<\varrho\leq r\leq R\leq \bar R$. 
Now, given $\mu=f\in L^q_\loc (\Om)$ for some $q>Q$, then by H\"older's inequality we have 
$$ \frac{|\mu|(B_r)}{r^{Q-1}} = \frac{1}{r^{Q-1}}\int_{B_r}|f|\dx \leq  \frac{|B_r|^{1-1/q}}{r^{Q-1}} \Big( \int_{B_r}|f|^q\dx\Big)^\frac{1}{q}\leq  c\, r^{1-\frac{Q}{q}}\|f\|_{L^q} .$$ 
Letting $\delta\leq (1-Q/q)/(p-1)$, the above implies that $w_\mu(r) \leq c\, r^\delta \|f\|_{L^q}^{1/(p-1)}$, and 
$0<\delta<1$ since $q>Q$ and $p\geq 2$. The same upper bound is also satisfied by the Wolff potential due to 
\eqref{eq:denswolf}. Furthermore, we assume $\bar\eps<2\alpha/p$ and 
$\delta<2\alpha/p-\bar\eps$ so that using \eqref{eq:Aml} again for the last term of \eqref{hl1} 
and a further reduction $L' \bar R^{\alpha-p(\delta +\bar\eps)/2}< 1$, leads to
\begin{equation*}
\begin{aligned}
\int_{B_\varrho}|\X u - (\X u)_{B_\varrho}|\dx \leq c 
\Big(\frac{\varrho^{Q+\beta}}{r^{\beta+\bar\eps}} + r^{Q+\delta}\Big)
\bigg[\intav_{B_R}(|\X u|+s)\dx + \|f\|_{L^q}^{1/(p-1)}\bigg], 
\end{aligned}
\end{equation*}
for every $0<\varrho\leq r\leq R\leq \bar R$. For some $0<\kappa<1$ we rewrite the above with the choice $r= \varrho^\kappa$ to have 
\begin{equation*}
\begin{aligned}
\int_{B_\varrho}|\X u - (\X u)_{B_\varrho}|\dx &\leq c \big(\varrho^{Q+(1-\kappa)\beta-\kappa\bar\eps} + \varrho^{\kappa(Q+\delta)}\big)
\bigg[\intav_{B_R}(|\X u|+s)\dx + \|f\|_{L^q}^{1/(p-1)}\bigg]\\
&\leq c\varrho^{Q+\gamma} \bigg[\intav_{B_R}(|\X u|+s)\dx + \|f\|_{L^q}^{1/(p-1)}\bigg],
\end{aligned}
\end{equation*}
where the latter inequality follows when $Q+\gamma \leq \min\{Q+(1-\kappa)\beta-\kappa\bar\eps\,,\, \kappa(Q+\delta)\}$; 
indeed we can make sure that this is true with the choice 
of $\kappa = \kappa(\gamma)$ such that 
$$ \frac{Q+\gamma}{Q+\delta} \,\leq\, \kappa \,\leq\, \frac{\beta-\gamma}{\beta+\bar\eps},$$
for any $0<\gamma< \beta\delta/(Q+\beta+\delta +\bar\eps)$. Also, note that if $\gamma, \bar\eps$ are small enough, $\kappa=\kappa(\gamma)$ can be chosen close enough to  $1$ and we can make sure $\varrho^\kappa\leq R$, whenever $0<\varrho<R$. 
Thus, we have obtained 
$$ \intav_{B_\varrho}|\X u - (\X u)_{B_\varrho}|\dx \leq c\varrho^\gamma
 \bigg[\intav_{B_R}(|\X u|+s)\dx + \|f\|_{L^q}^{1/(p-1)}\bigg],$$
for any $0<\varrho< R\leq \bar R$ 
and the proof is complete. 
\end{proof}

\subsection*{Acknowledgments}
S.M. was partially supported by 
%Analyysin ja dynamiikan huippuyksikk\"o (The Centre of Excellence in Analysis and Dynamics Research) and 
"Variationaaliset integraalit geometr" and GHAIA Marie Sk\l{}odowska-Curie grant agreement No 777822 under Horizon 2020. Y.S. is partially supported by the Simons foundation. This work was initiated at the occasion of a visit of S.M. at Johns Hopkins University, whose hospitality is acknowledged. 
The authors are thankful to Agnid Banerjee for fruitful discussions and suggestions.
%and to the anonymous referee for a careful review.

%******************************************************                                     BIB SETTINGS

\bibliographystyle{plain}
\bibliography{MyBib}

\def\cprime{$'$} \def\cprime{$'$} \def\cprime{$'$}
\begin{thebibliography}{10}

\bibitem{Boc-Gal}
Lucio Boccardo and Thierry Gallou\"{e}t.
\newblock Nonlinear elliptic and parabolic equations involving measure data.
\newblock {\em J. Funct. Anal.}, 87(1):149--169, 1989.

\bibitem{Bonfig-Lanco-Ugu}
A.~Bonfiglioli, E.~Lanconelli, and F.~Uguzzoni.
\newblock {\em Stratified {L}ie groups and potential theory for their
  sub-{L}aplacians}.
\newblock Springer Monographs in Mathematics. Springer, Berlin, 2007.

\bibitem{Camp}
Sergio Campanato.
\newblock Equazioni ellittiche del {${\mathrm{II}^\circ }$} ordine e spazi
  {${\mathcal L}^{(2,\lambda )}$}.
\newblock {\em Ann. Mat. Pura Appl. (4)}, 69:321--381, 1965.

\bibitem{C-D-G}
Luca Capogna, Donatella Danielli, and Nicola Garofalo.
\newblock An embedding theorem and the {H}arnack inequality for nonlinear
  subelliptic equations.
\newblock {\em Comm. Partial Differential Equations}, 18(9-10):1765--1794,
  1993.

\bibitem{C-D-S-T}
Luca Capogna, Donatella Danielli, Scott~D. Pauls, and Jeremy~T. Tyson.
\newblock {\em An introduction to the {H}eisenberg group and the
  sub-{R}iemannian isoperimetric problem}, volume 259 of {\em Progress in
  Mathematics}.
\newblock Birkh\"auser Verlag, Basel, 2007.

\bibitem{Dib}
E.~DiBenedetto.
\newblock {$C^{1+\alpha }$} local regularity of weak solutions of degenerate
  elliptic equations.
\newblock {\em Nonlinear Anal.}, 7(8):827--850, 1983.

\bibitem{Duz-Min}
Frank Duzaar and Giuseppe Mingione.
\newblock Gradient estimates via non-linear potentials.
\newblock {\em Amer. J. Math.}, 133(4):1093--1149, 2011.

\bibitem{Folland-Stein--book}
G.~B. Folland and Elias~M. Stein.
\newblock {\em Hardy spaces on homogeneous groups}, volume~28 of {\em
  Mathematical Notes}.
\newblock Princeton University Press, Princeton, N.J.; University of Tokyo
  Press, Tokyo, 1982.

\bibitem{Gia}
Mariano Giaquinta.
\newblock {\em Multiple integrals in the calculus of variations and nonlinear
  elliptic systems}, volume 105 of {\em Annals of Mathematics Studies}.
\newblock Princeton University Press, Princeton, NJ, 1983.

\bibitem{Gia-Giu--sharp}
Mariano Giaquinta and Enrico Giusti.
\newblock Sharp estimates for the derivatives of local minima of variational
  integrals.
\newblock {\em Boll. Un. Mat. Ital. A (6)}, 3(2):239--248, 1984.

\bibitem{Giu}
Enrico Giusti.
\newblock {\em Direct methods in the calculus of variations}.
\newblock World Scientific Publishing Co., Inc., River Edge, NJ, 2003.

\bibitem{Jerison}
David Jerison.
\newblock The {P}oincar\'e inequality for vector fields satisfying
  {H}\"ormander's condition.
\newblock {\em Duke Math. J.}, 53(2):503--523, 1986.

\bibitem{Kilp-Mal1}
Tero Kilpel\"{a}inen and Jan Mal\'{y}.
\newblock The {W}iener test and potential estimates for quasilinear elliptic
  equations.
\newblock {\em Acta Math.}, 172(1):137--161, 1994.

\bibitem{Kind-Stamp}
David Kinderlehrer and Guido Stampacchia.
\newblock {\em An introduction to variational inequalities and their
  applications}, volume~31 of {\em Classics in Applied Mathematics}.
\newblock Society for Industrial and Applied Mathematics (SIAM), Philadelphia,
  PA, 2000.
\newblock Reprint of the 1980 original.

\bibitem{Lewis}
John~L. Lewis.
\newblock Regularity of the derivatives of solutions to certain degenerate
  elliptic equations.
\newblock {\em Indiana Univ. Math. J.}, 32(6):849--858, 1983.

\bibitem{Lieb--gen}
Gary~M. Lieberman.
\newblock The natural generalization of the natural conditions of
  {L}adyzhenskaya and {U}ral\cprime tseva for elliptic equations.
\newblock {\em Comm. Partial Differential Equations}, 16(2-3):311--361, 1991.

\bibitem{Muk-Zhong}
S.~{Mukherjee} and X.~{Zhong}.
\newblock {$C^{1,\alpha}$-Regularity for variational problems in the Heisenberg
  group}.
\newblock {\em \url{https://arxiv.org/abs/1711.04671}}, to appear in Anal. PDE.
  2021.

\bibitem{Muk0}
Shirsho Mukherjee.
\newblock {On local Lipschitz regularity of {Q}uasilinear Elliptic equations in
  the {H}eisenberg {G}roup}.
\newblock {\em \url{https://arxiv.org/abs/1804.00751}}, 2018.

\bibitem{Tolk}
Peter Tolksdorf.
\newblock Regularity for a more general class of quasilinear elliptic
  equations.
\newblock {\em J. Differential Equations}, 51(1):126--150, 1984.

\bibitem{ZGold}
Anna Zatorska-Goldstein.
\newblock Very weak solutions of nonlinear subelliptic equations.
\newblock {\em Ann. Acad. Sci. Fenn. Math.}, 30(2):407--436, 2005.

\bibitem{Zhong}
Xiao Zhong.
\newblock Regularity for variational problems in the {H}eisenberg {G}roup.
\newblock {\em \url{https://arxiv.org/abs/1711.03284}}, 2017.

\end{thebibliography}

\end{document}